\documentclass[12pt]{amsart}

\usepackage{amsfonts, amssymb, amscd}

\def\dual                 {{\vee}} 
\def\rk                 {{\rm rk}}

\def\GKZ			{{\rm GKZ}}

\def\ZZ                 {{\mathbb Z}} 
\def\PP                {{\mathbb P}} 
\def\RR                 {{\mathbb R}} 
\def\CC                 {{\mathbb C}}

\def\Aa    {{\mathcal A}}

\newtheorem{lemma}{Lemma}[section] 
\newtheorem{theorem}[lemma]{Theorem} 
\newtheorem{corollary}[lemma]{Corollary} 
\newtheorem{proposition}[lemma]{Proposition} 
\theoremstyle{definition} 
\newtheorem{definition}[lemma]{Definition} 
\newtheorem{example}[lemma]{Example}

\newtheorem{remark}[lemma]{Remark} 
\theoremstyle{remark} 
\newtheorem*{proof*}{Proof} 
 
\title[On the better behaved version of the GKZ ...]{On
the better behaved version of the GKZ hypergeometric  
system}
 
\author{Lev A. Borisov and R. Paul Horja} 
\address{Department of Mathematics \\ Rutgers University \\ 
Piscataway \\ NJ \\ 08854-8019 \\ USA \\{\tt borisov@math.rutgers.edu}} 
 \address{Department of Mathematics \\ Oklahoma State University \\ 
Stillwater\\ OK \\ 74078 \\ USA \\{\it Current address}:
Universit\"at Wien \\ Fakult\"at f\"ur Mathematik \\ 
A-1090 Wien \\ Austria\\{\tt richard.paul.horja@univie.ac.at}} 
 
\thanks{The first  
author was partially supported by the NSF grant DMS-1003445. The second author
was partially supported by the NSA grant MDA904-10-1-0190.}

\begin{document} 
 
\begin{abstract} We consider a version of the generalized hypergeometric
system introduced by Gelfand, Kapranov and Zelevinski (GKZ) suited for the 
case when the underlying lattice is replaced by a finitely generated
abelian group. In contrast to the usual GKZ hypergeometric system, the rank of the better 
behaved GKZ hypergeometric system is always the expected one. 
We give largely self-contained proofs of many properties of this system.
The discussion 
is intimately related to the study of the 
variations of Hodge structures of 
hypersurfaces in algebraic tori.
\end{abstract} 
 
\maketitle

\section{Introduction} 
The seminal paper of Gelfand, Kapranov and Zelevinsky 
\cite{GKZ} introduced generalizations of the classical hypergeometric function
as solutions to certain systems of partial differential equations 
defined by combinatorial data. 
These GKZ hypergeometric functions have been studied extensively in the subsequent years. In particular, they appear naturally in the study of mirror symmetry for hypersurfaces and complete intersections in toric varieties,
see \cite{BatyrevVanstratten, HLY, stienstra}. Unfortunately, the version of the GKZ hypergeometric system that is most suitable to mirror symmetry applications has received scant attention in the more general works on the GKZ systems, such as \cite{MMW,SST}. The disconnect 
between the general theory and the most interesting (from our viewpoint) applications makes the theory seem far more formidable than it really is and inhibits the development of mirror symmetry.

In our paper we aim to bridge this gap by considering the \emph{better behaved version} 
of the GKZ hypergeometric system. The term "better behaved" is used to indicate the absence of the rank jumping phenomena.
It is also precisely the version of the system that is most useful for mirror symmetry applications, see for example \cite{Borisov2}.

We deliberately try to make our arguments as self-contained as possible.
In order to make the subject more accessible, we avoid any $D$--module technology.
In the process we give a more algebraic treatment of the cohomology
spaces of nondegenerate hypersurfaces in algebraic tori that were originally studied by
Batyrev in \cite{Batyrev}.

Let us first recall the original definition of GKZ hypergeometric system.
Let 
$\Aa= \{ v_1, \ldots, v_n \}$ be a set of vectors in the
lattice $N \cong \ZZ^d$ such that the elements of $\Aa$ 
generate the lattice as an abelian group, and there exists a
group homomorphism ${\rm deg} : N \to \ZZ$ such that ${\rm deg}(v)=1$ for any element
$v \in \Aa.$ Let $L \subset \ZZ^n$ denote the lattice of integral relations 
among the elements of $\Aa$ consisting of vectors $l=(l_i) \in \ZZ^n$ 
such that $l_1 v_1 + \ldots + l_n v_n =0.$ 

For any parameter 
$\beta \in N \otimes \CC,$  
Gelfand, Kapranov and Zelevinsky \cite{GKZ}
considered  a
system of differential equations on the function $\Phi(x),$
$x=(x_1, \ldots, x_n) \in \CC^n,$ consisting of the
binomial equations
$$
\Big( \prod_{i, l_i >0} (\partial_i)^{l_i}
- \prod_{i, l_i < 0} (\partial_i)^{-l_i} \Big)
\Phi =0, \ l \in L, 
$$
and the linear equations
$$
\Big( \sum_{i=1}^n \mu(v_i) x_i \partial_i \Big) \Phi= \mu(\beta) \Phi, \
\text{for all $\mu \in M = {\rm Hom} (N, \ZZ).$}
$$
Gelfand, Kapranov and Zelevinsky
showed that this system is holonomic, so the number
of linearly independent solutions at a generic point is finite.
Following Batyrev's observation \cite[Section 14]{Batyrev}
that the periods of a Calabi--Yau
hypersurface in a projective toric variety satisfy a
GKZ system, its study gained further prominence
in connection with mirror symmetry phenomena
and algebra geometric applications. 

The rank of the GKZ system (the dimension of its 
solution set at a generic point) and the solution set itself
have also been the subject of numerous studies. 
Its expected dimension 
is equal to the normalized volume 
of the convex hull $\Delta$ of the elements of the
set $\Aa.$ However, 
unless the toric ideal
associated to $\Aa$ is Cohen--Macaulay,
there are non-generic values of $\beta$ for which 
the rank jumps. This rank discrepancy has been thoroughly
investigated by many authors (see, for example, 
Adolphson \cite{A}, Saito, Sturmfels and Takayama \cite{SST}, 
Cattani, Dickenstein and Sturmfels \cite{CDS}) 
and a quite definitive explanation for it
has been obtained 
in the work of Matusevich, Miller 
and Walther \cite{MMW}.

In the present work, we focus on a {\it better
behaved} version of the GKZ system whose
space of solutions always has the expected
dimension. We frame the definition 
in a context where the lattice is replaced
by a finitely generated abelian group $N,$ and 
the set $\Aa$ is replaced by an 
$n$-tuple $\Aa=(v_1,\ldots,v_n)$ of elements of $N,$ 
with possible repetitions.
Given 
a parameter 
$\beta$ in $N \otimes \CC,$ the better 
behaved GKZ system consists of the equations
$$
\partial_i\Phi_c=\Phi_{c+v_i},~{\rm~for~all~}c\in K,~i\in\{1,\ldots,n\} 
$$
and the linear equations
$$
\sum_{i=1}^n \mu(v_i) x_i\partial_i \Phi_c=\mu(\beta-c) \Phi_c, 
{\rm~~for~all~}\mu\in M,~c\in K. 
$$
A solution to the better behaved GKZ system 
is then a
sequence of functions of $n$ variables 
$\big(\Phi_c(x_1, \ldots, x_n)\big)_{c \in K},$
where $K$ is the preimage under the map $N \to N \otimes \RR$
of the cone 
$K_{\RR}$ generated by the images of the elements $v_i$
in $N \otimes \RR.$

When $N$ is a lattice and $\Aa$ is 
a finite subset subject to the the hyperplane condition, the better behaved GKZ equations on $\Phi_0$ 
imply the usual GKZ equations on that function. 
The generalization presented in our work 
fits in the general context of ideas where the 
usual combinatorial framework of
toric geometry is extended 
from toric varieties and their fans to 
that of toric Deligne-Mumford stacks and stacky fans 
provided in the work of
Borisov, Chen and Smith \cite{BCS}. 

We now briefly discuss the contents of this paper.
In section \ref{sectionGKZ}, we give the precise definition
of the better behaved GKZ system. In section 
\ref{section:spacesbbgkz}, we relate
 the spaces of solutions to the logarithmic Jacobian
rings (Definition \ref{def}). In particular, we prove 
that the spaces of solutions have indeed the 
expected dimension, namely the product of 
the normalized volume of the polytope $\Delta$ and
the torsion order of the abelian group $N.$
In section \ref{section:gamma}, we investigate the effects of torsion in $N$
and repetitions among the elements $v_i$.
In section \ref{added} we study the restriction map from the solution space of GKZ 
for the cone to that for its interior. This restriction map is familiar in 
mirror symmetry. For example, in the case of the quintic hypersurface in $\PP^4,$ it amounts 
to considering the restriction map  
from the space of periods associated to the mirror Landau--Ginzburg model
of $\PP^4$ to the periods of the Calabi--Yau mirror quintic. 
The results of sections \ref{added} 
are intimately related to the work
of Batyrev \cite[section 8]{Batyrev}, 
but our treatment is more algebraic and 
self-contained.
In the last section, we discuss some variations of the main 
construction, as well as some
open problems related to it.

{\bf Acknowledgements.} Upon learning about our construction, 
in a letter to one of the authors, Alan Adolphson \cite{Adolph} 
wrote us that he obtained a similar definition for a 
generalization of the GKZ system in the case of
a lattice $N.$ Hiroshi Iritani informed us that 
in a recent preprint \cite{Iri}, the better
behaved GKZ system appears as a natural ingredient 
in his study of the quantum $D$-module associated 
to a toric complete intersection and the periods of its mirror.
We would also like to thank Vladimir Retakh
for a useful reference and the referee for the 
careful and thoughtful reading of the text. We sketch an alternative point 
of view on some of the constructions presented
in this paper in Remarks \ref{rem:ref1},  \ref{rem:ref2} and \ref{rem:ref3},
following the referee's suggestions.

\section{The usual and the better behaved versions of 
the GKZ 
hypergeometric system}\label{sectionGKZ}
 
Throughout this paper, we will use the following notations. 
We are given a finitely generated abelian group $N$, and 
an $n$-tuple $\Aa=(v_1,\ldots,v_n)$ of elements of $N$. 
We will denote by $M$ the free abelian group ${\rm Hom}(N,\ZZ)$. 
We will assume that there exists an element ${\rm deg}\in M$  
such that ${\rm deg}(v_i)=1$ for all $i$.   
We will denote by $\Delta$ the convex hull of the set of $v_i$ 
in $N\otimes \RR$ and by $K_\RR$ the cone 
$\RR_{\geq 0} \Delta$. 
We will denote by $K$ the preimage 
of $K_\RR$ in $N$ under the natural map $\pi: N\to N\otimes \RR$ 
and by $[c],$ for $c \in K,$ the corresponding elements in the 
semigroup ring $\CC[K]$ or its variants that will be  
used below. We will further assume that $\pi(v_i)$ span the 
lattice $\pi(N)$ as a group. 
The finite abelian group ${\rm tors} (N)$ is the torsion
part of $N$ and 
$|{\rm tors}(N)|$ its order. For $N$ torsion-free, 
we set $|{\rm tors}(N)|=1.$ 
 
The version of the GKZ hypergeometric  
system associated to a fixed parameter $\beta\in N\otimes \CC$
which will be the central object of study of this paper
is then defined as follows:
\begin{definition}\label{trueGKZdef} 
Consider the following system of partial differential equations 
on sequences of functions of $n$ variables 
$\big(\Phi_c(x_1, \ldots, x_n)\big)_{c \in K}:$
\begin{equation}\label{1} 
\partial_i\Phi_c=\Phi_{c+v_i},~{\rm~for~all~}c\in K,~i\in\{1,\ldots,n\} 
\end{equation} 
\begin{equation}\label{2} 
\sum_{i=1}^n \mu(v_i) x_i\partial_i \Phi_c=\mu(\beta-c) \Phi_c, 
{\rm~~for~all~}\mu\in M,~c\in K. 
\end{equation} 
We will call this system the {\em better behaved GKZ}  
and will denote it by $\GKZ(\Aa,K;\beta)$.  
\end{definition} 
In order to simplify our notation, 
we will denote a solution to the better behaved GKZ by 
$\Phi_K(x_1, \ldots, x_n)$.  

\begin{remark}\label{rem:ref1}
The solution 
$\Phi_K(x_1, \ldots, x_n)$
can 
be viewed as a function
from the  $\CC$-vector space 
$\CC^{\Aa}:= {\rm Maps} (\Aa, \CC)$  
with values in $\CC^K:= {\rm Maps} (K, \CC),$
the linear dual of the $\CC$-vector space 
$\CC[K].$
\end{remark}

\begin{remark}\label{finite} 
It is clear that one can reformulate the above system as  
a system of PDEs on a finite collection of  
functions of $(x_1,\ldots,x_n)$. Indeed, the set $K_{\rm prim}$ 
of elements $v\in K$ such that $v-v_i\not\in K$ for all $i$ 
is finite. The functions $\Phi_c$ for $c\in K_{\rm prim}$  
then determine the rest of $\Phi_c$. In fact, the number of  
PDEs can also be made finite, in view of the following. 
The relations \eqref{2} for $c\in K$, together with 
the relation \eqref{1} implies the relations \eqref{2} for $c+v_i$. 
Consequently, one only needs to use \eqref{2} for $c\in K_{\rm prim}$. The relations \eqref{1} can then be restated as 
\begin{equation}\label{3} 
\Big(\prod_{i=1}^{n} \partial_i^{k_i}\Big)\Phi_{c_1} =  
\Big(\prod_{i=1}^{n} \partial_i^{l_i}\Big)\Phi_{c_2} 
\end{equation} 
for all $k_i,l_i\in \ZZ_{\geq 0}$ such that  
$$c_1+\sum_i k_i v_i = c_2 + \sum_i l_i v_i$$ 
and $c_1,c_2\in K_{\rm prim}$.  
To see that \eqref{3} follows from a finite number 
of relations of this type, note that they correspond 
to the $\CC$-basis of the module over the polynomial  
ring $\CC[\partial_1,\ldots,\partial_n]$ which is the kernel of  
the natural map 
$$ 
\CC[K_{\rm prim}] \otimes \CC[\partial_1, \ldots, \partial_n] 
\to  
\CC[K] 
$$ 
which sends $\partial_i\to [v_i]$. Since $K_{\rm prim}$ is 
finite, this kernel is a Noetherian module, thus a finite subset of  
\eqref{3} generates the rest. 
\end{remark} 
 
\begin{remark} 
The usual GKZ hypergeometric system coincides with 
$\GKZ(\Aa,K;\beta)$ if $N$ has no torsion and $v_i$ generate $K$ 
\emph{as a semigroup}. Indeed, then $K_{\rm prim}=\{0\}$, 
\eqref{2} leads to the linear equations of \cite{GKZ} 
and \eqref{1} leads to  
$$ 
\Big(\prod_{i=1}^{n} \partial_i^{k_i}\Big)\Phi_{0} =  
\Big(\prod_{i=1}^{n} \partial_i^{l_i}\Big)\Phi_{0} 
$$ 
whenever $\sum_{i}(k_i-l_i)v_i=0$, which are the binomial 
relations of \cite{GKZ}. 
\end{remark}

\begin{remark}  
The $n$-tuple $\Aa$ of elements of $N$ is allowed to contain 
repeated elements. As one can see from the PDEs defining the  
better behaved GKZ system, the effect of having  
$v_i=v_j$ for some $i \not= j,$ is that all functions $\Phi_c$  
depend on $x_i + x_j$. 
\end{remark} 
 
\begin{example} Let $N= \ZZ \oplus \ZZ / 2\ZZ$ 
and $\Aa = (v_1, v_2),$ with  
$v_1= (1,0),$ $v_2=(1,1).$ Let $\beta$ be an 
element in $N \otimes \CC \cong \CC.$ 
The solution space 
of the better behaved GKZ system is isomorphic 
to the space of pairs of functions $\Phi_{(0,0)}(x_1,x_2), \Phi_{(0,1)}(x_1,x_2)$ 
satisfying the equations 
$$ 
\partial_1 \Phi_{(0,0)} = \partial_2 \Phi_{(0,1)}, \quad  
\partial_2 \Phi_{(0,0)} = \partial_1 \Phi_{(0,1)},$$ 
$$ 
(x_1 \partial_1 + x_2 \partial_2) \Phi_{(0,0)} = \beta \Phi_{(0,0)}, \quad 
(x_1 \partial_1 + x_2 \partial_2) \Phi_{(0,1)} = \beta \Phi_{(0,1)}. 
$$ 
The first pair of equations implies that both  
functions $\Phi_{(0,0)}$ and $\Phi_{(0,1)}$ satisfy 
the wave equation. It follows that 
$$\Phi_{(0,0)} (x_1,x_2)=a(x_1+x_2) + b(x_1 - x_2),$$$$
\Phi_{(0,1)}(x_1,x_2)= a(x_1+x_2) - b(x_1-x_2),$$ 
for some arbitrary functions $a,b.$ The second  
pair of equations implies then that 
$$ 
(x_1+x_2) a^\prime(x_1+x_2)= \beta a(x_1+x_2), \
(x_1-x_2) b^\prime(x_1-x_2)= \beta b(x_1-x_2). 
$$ 
It follows that 
$a(x_1+x_2) = A (x_1+x_2)^\beta$ and  
$b(x_1 -x_2)= B (x_1- x_2)^\beta,$ for some arbitrary complex 
constants $A,B.$ Hence the better behaved GKZ system 
has a two-dimensional solution space. Note 
that the discriminant locus of the system consists 
of the reducible curve $x_1^2- x_2^2=0$ in $\CC^2.$  
\end{example}

\begin{definition}\label{GKZ(S)} 
For any subset $S$ of $N$ which is closed under  
the additions of $v_i$ we can define the system 
$\GKZ(\Aa,S;\beta)$ as in Definition \ref{trueGKZdef}, 
but with $c\in S$ rather than $c\in K$.  
\end{definition} 
 
\begin{remark} 
If $N$ has no torsion, then  
the usual version of GKZ is equivalent to $\GKZ(\Aa,S;\beta)$ 
for $S$ the subsemigroup of $K$ generated by $v_i$. 
The fact that $\GKZ(\Aa,K;\beta)$ is better behaved than  
the usual GKZ is then related to the fact that the semigroup 
algebra $\CC[K]$ is always Cohen-Macaulay, whereas 
$\CC[S]$ need not be so. 
\end{remark}

\begin{remark}\label{rem:ref2}
In the language of Remark \ref{rem:ref1}, 
the better behaved GKZ system can be viewed as a vectorized
and more general version of the classical GKZ system. 
As such, 
the first set of equations defining the better behaved GKZ system
states that the better behaved GKZ solution
$\Phi_K$ identifies 
the linear action of the derivations with the 
semigroup action by translations by elements of
$K$ on the set $K.$ 

The second set of equations can be obtained by 
considering the action of $M \otimes \CC^{\times}$
on the vector valued function 
$(\Phi_c((x_v)_{v \in \Aa}))_{c\in K}$
as follows. 
For a chosen basis $\mu_1, \ldots, \mu_m$ 
of $M,$ and an element 
$\sum_{j=1}^m \mu_j \otimes t_j$ in $M \otimes \CC^{\times},$
the action is given by 
$$
\big( 
\sum_{j=1}^m \mu_j \otimes t_j \big) \cdot
\big(\Phi_c((x_v)_{v \in \Aa})_{c\in K}\big) :=
\big(
\prod_{j=1}^m t_j^{\mu_j (c)}
\Phi_c ( (\prod_{j=1}^m t_j^{\mu_j(v)} x_v)_{v \in \Aa})
\big)_{c \in K}.
$$
A direct calculation shows that, as in the case of the classical GKZ system, 
the second set of better behaved GKZ equations can be naturally stated 
by considering the 
infinitesimal (i.e. Lie algebra) version of this group
action. 
\end{remark}

\begin{remark}\label{rem:ref3}
One can view a solution $\Phi$ to a better behaved GKZ system as 
a function on the space of $x$ and the semigroup ring $\CC[K]$
given by 
$$
(x_1,\ldots,x_n; \sum_{c\in K} a_c[c] )\mapsto \sum_{c\in K}
a_c\Phi_c(x_1,\ldots,x_n).
$$
In these notations $\frac \partial{\partial x_i} \Phi({\bf x};g) = \Phi({\bf x};[v_i]g)$.
\end{remark}

\section{Spaces of solutions of the  
better behaved GKZ and the  
logarithmic Jacobian ring}\label{section:spacesbbgkz} 
 
Let $(x_1,\ldots,x_n) \in \CC^n.$ 
We introduce a non-degeneracy notion for a 
degree one element $f=\sum_{i=1}^n x_i [v_i]$ of $\CC[K]$ 
which is closely related to the one used 
by Batyrev \cite{Batyrev} in the non-torsion case 
(see for example \cite[Theorem 4.8]{Batyrev}).  
 
\begin{definition} The degree one element  
$f=\sum_{i=1}^n x_i [v_i]$ of $\CC[K]$ is said to be 
{\em non-degenerate} if the logarithmic derivatives 
$\sum_i x_i \mu_j(v_i)[v_i]$ form a regular sequence 
in $\CC[K]$ for a basis $\mu_j,$ $1 \leq j \leq \rk M,$ 
of $M.$ 
\end{definition} 
  
\begin{proposition} 
For a generic choice of $f=\sum_{i=1}^n x_i[v_i]$ 
and any basis $(\mu_j)$ of $M$ the  
log-derivatives $f_j=\sum_i x_i \mu_j(v_i)[v_i]$ of $f$ 
give a regular sequence in $\CC[K]$.  
Equivalently, the Koszul complex induced by the  
elements $f_j$ 
\begin{equation}\label{Koszul} 
0\to \ldots \to \wedge^2 M\otimes \CC[K]\to M\otimes \CC[K] \to\CC[K]\to R(f,K)\to 0 
\end{equation} 
is exact. 
\end{proposition} 
 
\begin{proof} 
If $N$ has no torsion, the result is 
\cite[Proposition 3.2]{Borisov}.  
The  
Koszul complex reformulation is standard.
If $N$ has torsion, the result appears to be new,  
but perhaps not particularly unexpected.  
In order to prove it, note that the ring
$\CC[K]$ is the direct sum of
$|{\rm tors} N|$ copies of 
$\CC[\pi(K)],$
where $\pi : K \to K \otimes \RR$ is the
natural map. Then
the regularity of the sequence
needs to be checked at each individual
copy of $\CC[\pi(K)]$ where it
follows again from the non-torsion result.  
\end{proof} 
 
\begin{definition} \label{def}
The ring $R(f,K),$ defined 
by the Koszul complex of the 
previous proposition, 
is called the {\em logarithmic Jacobian  
ring} associated to $f$ and $K.$ 
\end{definition} 
 
\begin{corollary} 
The dimension of the $\CC$-vector space  
$R(f,K)$ 
is equal to $vol(\Delta) \cdot | {\rm tors}(N)|,$ 
where $vol(\Delta)$ is the  
normalized volume of the polytope $\Delta$ in 
$N \otimes \RR,$ and $|{\rm tors}(N)|$ is the order of 
the torsion part of $N.$
\end{corollary} 
 
\begin{proof} 
The dimension of the  $\CC$-vector space  
$R(f,K)$ is equal to the product of 
$(\rk N -1)! \cdot |{\rm tors}(N)|$ and 
the leading coefficient of the  
Hilbert polynomial of the graded ring 
$\CC[K \otimes \ZZ].$ But it is well 
known that this leading coefficient 
is the quotient of the normalized 
volume of $\Delta$ by $(\rk N -1)!.$  
\end{proof}

The complex \eqref{Koszul} is graded with finite-dimensional 
graded components. We can dualize it component-wise
to get another graded exact complex with
finite-dimensional graded components
\begin{equation}\label{Kdual} 
0\to R(f,K)^\vee \to \CC[K] \to N\otimes \CC[K] \to \wedge^2N\otimes \CC[K] 
\to \ldots \to 0 .
\end{equation} 
We will naturally identify the graded dual of $\CC[K]$ with 
itself, since each graded component of  
$\CC[K]$ has a natural basis. 
 
The complex \eqref{Kdual} allows us to give the following  
description of the vector space $R(f,K)^\vee$. 
\begin{proposition} 
The space $R(f,K)^\vee$ is the set of elements 
$\sum_{c \in K} \lambda_c [c]$ in  
$\CC[K]$ such that 
the linear equations in $N \otimes \CC$  
\begin{equation}\label{rfkdual} 
\sum_{i=1}^n x_i\lambda_{c+v_i} v_i = 0 
\end{equation} 
hold for all $c\in K$.
\end{proposition} 
 
\begin{proof} 
The result follows from the observation that  
the dual of the map $M\otimes \CC[K] \to\CC[K]$ in the  
Koszul complex \eqref{Koszul} is the map  
$\CC[K] \to N\otimes \CC[K]$ in the dual 
complex \eqref{Kdual} 
given by 
$$ 
\sum_{c \in K} \lambda_c [c] \mapsto  
\sum_{c \in K} \sum_{i=1}^n x_i \lambda_{c+v_i} v_i \otimes [c]. 
$$ 
\end{proof} 
 
\begin{remark}\label{rmk} 
Note that equations (\ref{rfkdual}) can be solved degree-by-degree 
and will have no nontrivial solutions for $\deg(c)>\rk N$. 
Indeed, the exactness of the complex \eqref{Kdual} 
implies that the Hilbert-Poincar\'{e} series  
of the kernel of the map $\CC[K] \to N\otimes \CC[K]$ 
is a polynomial of degree at most $\rk N.$ 
\end{remark} 
 
Let us now consider the  
solutions to $\GKZ(\Aa,K;\beta)$. 
\begin{theorem}\label{key} 
The space of analytic
solutions to $\GKZ(\Aa,K;\beta)$ in a 
neighborhood of a generic $f$ is isomorphic  
to the space of elements 
$(\lambda_c)_{c \in K}$ 
in $\CC^K= {\rm Maps} (K, \CC)$ 
such that 
the linear equations in $N \otimes \CC$  
\begin{equation}\label{GKZlambdas} 
\sum_{i=1}^n x_i\lambda_{c+v_i} v_i = \lambda_c(\beta-c) 
\end{equation} 
hold for all $c \in K.$ 
\end{theorem} 
 
\begin{proof} 
In one direction, if we have a solution $(\Phi_c), c\in K$,  
then $\lambda_c=\Phi_c(x_1,\ldots, x_n)$ clearly satisfies  
\eqref{GKZlambdas}. In fact, this map from the space of 
solutions of $\GKZ(\Aa,K;\beta)$ to the space of solutions 
of \eqref{GKZlambdas} is clearly injective in view of
Taylor's formula, since knowing  
all $\Phi_c(x_1,\ldots, x_n)$ implies the knowledge of 
all the partial derivatives of all $\Phi_c$ at $(x_1,\ldots,x_n)$ 
in view of equation \eqref{1}. 
 
In the other direction,  
suppose that we have a solution $(\lambda_c)$ of 
\eqref{GKZlambdas}. Then equation \eqref{1} and 
Taylor formula force us 
to have  
\begin{equation}\label{Taylor} 
\Phi_c(z_1,\ldots,z_n) = \sum_{(l_1,\ldots,l_n)\in \ZZ_{\geq 0}^n} 
\lambda_{c+\sum_i l_i v_i} 
\prod_{i=1}^n \frac {(z_i-x_i)^{l_i}}{l_i!} 
\end{equation} 
for all $c\in K$. 
It remains to show that the above series converges 
absolutely and uniformly in $c\in K$ and ${\bf z}$ in a neighborhood of $(x_1,\ldots, x_n)$. Observe that it suffices to show uniform convergence for a  
fixed $c\in K_{\rm prim}$, since the  
partial derivative of a Taylor series will converge in  
the same neighborhood and $K_{\rm prim}\subset K$ 
is a finite set. From now on we fix $c=c_0$. 
 
We claim that there exists a constant $C_1 \in \RR$ 
such that  
\begin{equation}\label{estimate} 
\vert \lambda_{c_0+\sum_{i}l_i v_i}\vert \leq  
C_1^{(\sum_{i=1}^n l_i)} 
(\sum_{i=1}^n{l_i})! 
\end{equation} 
for all nonzero $(l_1,\ldots,l_n)\in \ZZ_{\geq 0}^n$. This is  
easily seen to be equivalent to the existence of a constant 
$C_2 \in \RR$ such that  
\begin{equation}\label{estimate2} 
\vert\lambda_{d}\vert \leq C_2^{\deg d} (\deg d)! 
\end{equation} 
for all $d$ with sufficiently high $\deg d$. 
 
Define $\Lambda_k=\displaystyle\max_{d,\deg d=k}\vert \lambda_d\vert$. 
To prove \eqref{estimate2} it suffices to show that there exists $C_3\in \RR$ such that $\Lambda_{k+1}\leq C_3k\Lambda_k$ for all sufficiently large $k$.  
 
The ideal $I$ of $\CC[K]$ generated by logarithmic derivatives 
of $f$ contains $[d]$ for all $d$ of $\deg d=\rk N +1$. It is easy 
to see that every $d_1$ of sufficiently high degree can be  
written as $d_1=d+d_2$ with $d,d_2\in K$ and $\deg d=\rk N+1$. 
We can write each $[d]$ of degree $\rk N+1$ as 
$$ 
[d] = \sum_{i=1}^n\sum_{j=1}^{\rk N} 
 x_i \mu_j(v_i) [v_i] t_{d,j} 
$$ 
for some $t_{d,j}\in \CC[K]_{\deg = \rk N}$ and some 
basis $(\mu_1,\ldots, \mu_{\rk N})$ of $M$.  
Consequently, for 
each  $d_1$ of sufficiently high degree we have  
that 
$$ 
[d_1] =  \sum_{i=1}^n\sum_{j=1}^{\rk N} 
 x_i \mu_j(v_i) [v_i] t_{d,j}[d_2] 
$$ 
for some $d_2$. By considering the maximum size of  
the coefficients of $t_{d,j}$ we observe 
that, for $\deg d_1=k+1,$  
$$ 
[d_1] = \sum_{i=1}^n\sum_{j=1}^{\rk N} 
  \sum_{d_3,\deg d_3= k} \beta_{d_3,j}  
 x_i\mu_j(v_i) [d_3+v_i] 
$$ 
with $\sum_{d_3,j}\vert\beta_{d_3,j}\vert$ bounded by a constant independent of  
$d_1$ and $k$.  
Equation \eqref{GKZlambdas} implies then that 
$$\sum_{d_3,j} \beta_{d_3,j} \lambda_{d_3}\mu_j(\beta-d_3) 
=\sum_{i,d_3,j} \beta_{d_3,j} x_i\lambda_{d_3+v_i} \mu_j(v_i) =  
\lambda_{d_1}.$$ 
Since $\sum_{d_3,j}\vert\beta_{d_3,j}\vert$ is bounded by  
a constant, $\vert \lambda_{d_3}\vert$ is bounded by $\Lambda_k$  
and $\mu(\beta-d_3)$ is bounded by a constant  
times $k$, we get $\vert \lambda_{d_1}\vert \leq C_3k\Lambda_k$ 
as required.  
 
This allows us to establish estimates \eqref{estimate2} and 
\eqref{estimate}.  
Since the multinomial coefficients 
$$\frac {(\sum_{i=1}^n l_i)!}{\prod_{i=1}^n l_i!}$$ 
are bounded by $n^{\sum_{i=1}^n l_i}$, the terms of the series 
\eqref{Taylor}  
are bounded by $\prod_{i=1}^n (nC_1)^{l_i}\vert z_i-x_i\vert ^{l_i}$. 
By making $\vert z_i-x_i\vert $ sufficiently small, the absolute 
convergence is obtained by comparing to a product of convergent 
geometric series. 
\end{proof} 
 
Having identified the space of solutions of $\GKZ(\Aa,K;\beta)$ 
in a neighborhood of $(x_1,\ldots, x_n)$  
with the space of solutions of the equations \eqref{GKZlambdas}, we can 
now consider a natural filtration on it. 
We define the subspaces $F_k$ of the  
space of solutions of $\GKZ(\Aa,K;\beta)$ 
in a neighborhood of a generic $(x_1,\ldots, x_n)$  
to be characterized by the fact that 
$\Phi_c(x_1,\ldots,x_n)=0$ for all $c$ with $\deg c <k$. 
We have that $F_0\supseteq F_1\supseteq F_2 \supseteq \cdots$. 
 
\begin{theorem}\label{main} 
The quotient $F_{k}/F_{k+1}$ is naturally isomorphic to 
the dual of the degree $k$ component of $R(f,K)$. 
\end{theorem} 
 
\begin{proof} 
The essential observation is that the equations \eqref{GKZlambdas} 
satisfied by the elements $\lambda_c$  
can be solved recursively in the degree of $c.$ 
Indeed, suppose that we have 
found $\lambda_d, \deg d\leq k,$ which satisfy  
\eqref{GKZlambdas} for all $c$, $\deg c\leq k-1$. 
In order to check that a solution exists for all $d$ of degree $k+1$, 
we need to check that  
$\sum_{c,\deg c=k} \lambda_c(\beta - c)[c]$ sits in the  
degree $k$ component of the  
image of the map 
$\CC[K]\to N\otimes \CC[K]$
of the complex \eqref{Kdual}. Since this is an exact complex, 
it suffices to check that it is in the kernel of the map 
$$ 
N\otimes \CC[K] \to \wedge^2N \otimes \CC[K]
$$
 of \eqref{Kdual}. The coefficient of its image at $[c_1]$ 
is given by the element of $\wedge^2N\otimes \CC[K]$ 
$$ 
\sum_{i=1}^n x_i\lambda_{c_1+v_i} v_i\wedge(\beta-c_1) 
=\lambda_{c_1}(\beta-c_1)\wedge(\beta-c_1)=0. 
$$ 
 
Observe that as we are solving recursively the  
equations \eqref{GKZlambdas}, the ambiguity at each 
step is precisely an element of the corresponding component 
of $R(f,K)^\vee$, which leads to the result. 
\end{proof} 
 
\begin{corollary} 
The space of solutions to the better behaved GKZ system is  
of the same dimension as $R(f,K)$.   
\end{corollary} 
 
\begin{remark}\label{same}
The same argument applies with obvious modifications 
when one replaces $K$ by its interior $K^\circ$. 
\end{remark} 
 
\begin{remark} 
The argument of this section is likely philosophically the  
same as the general arguments used in the theory of holonomic 
$D$-modules, but it has an advantage of being self-contained. 
\end{remark}

\section{Effects of torsion and repetitions} \label{section:gamma}
In this section we analyze the role played by the torsion part of the finitely
generated abelian group $N$ and by the possible repetitions that 
may appear in the $n$-tuple $\mathcal A.$
Let $\{ w_1, \ldots, w_m \} \subset N \otimes \RR$ be the set consisting
of the elements $\pi (v_i), 1 \leq i \leq n,$ in $N \otimes \RR$ where $\pi : N \to N \otimes \RR$
is the natural map. For each $j, 1 \leq j \leq m,$ let $I_j$ be the set of 
indices $i$ with $\pi(v_i)=w_j.$

Let $\rho : N \to \CC^{\times}$ be a multiplicative group character.
Define the map $p_{\rho} : \CC^n \to \CC^m$ by
\begin{equation}\label{def:mapchar}
p_\rho(x_1, \ldots, x_n):= (\sum_{i \in I_1} \rho(v_i) x_i, \ldots,
\sum_{i \in I_m} \rho(v_i) x_i).
\end{equation}
To a sequence of functions $(\Psi_{w}(z_1, \ldots, z_m))_{w \in \pi(K)},$
we associate a sequence of functions $(\Phi_c (x_1, \ldots, x_n))_{c \in K}$
such that, for any $c \in K,$ 
\begin{equation}\label{eq:torext}
\Phi_c (x_1, \ldots, x_n):=
\rho(c) \Psi_{\pi(c)} (p_{\rho}(x_1, \ldots, x_n)).
\end{equation}
What makes this definition useful is the following result. 

\begin{proposition} 
For any character $\rho \in {\rm Hom} (N, \CC^{\times}),$
if the function
$$
\Psi_{\pi(K)} (z_1, \ldots, z_m)=
\big(\Psi_w (z_1, \ldots,z_m)\big)_{w \in p(K)} 
$$
is a solution on an open set $U \subset \CC^m$
to the 
better behaved GKZ associated to $\beta \in N \otimes \CC$
defined by $\{ w_1, \ldots, w_m \}$ in $\pi(N),$
then the 
associated 
function 
$$\Phi_K (x_1, \ldots, x_n)=
\big(\Phi_c(x_1, \ldots, x_n)\big)_{c \in K}$$
is a solution on the open set $p^{-1}_\rho (U) \subset \CC^n$
to the
better behaved GKZ 
associated to $\beta \in N \otimes \CC$
defined by $(v_1, \ldots, v_n)$ in $N.$ 
\end{proposition}

\begin{proof} For any $c \in N,$ equation (\ref{eq:torext}) implies that
given some $v_i \in N$ and $w_j \in \pi(N)$ such that
$\pi (v_i)=w_j$ we have that
$$
\partial_i \Phi_c (x_, \ldots, x_n)
= \rho(c+v_i) \partial_j \Psi_{\pi(c)} 
(\sum_{i \in I_1} \rho(v_i) x_i, \ldots,
\sum_{i \in I_m} \rho(v_i) x_i).
$$
Since the functions $\Psi_{w}, w \in \pi(K),$ are solutions to the better behaved GKZ in 
$\pi(N),$ and $\pi(c)+w_j= \pi(c) + \pi(v_i)= \pi (c+v_i),$
we obtain indeed that 
$
\partial_i \Phi_c (x_, \ldots, x_n)= \Phi_{c+v_i}.
$
Similarly, for any $\mu \in M={\rm Hom} (N, \ZZ)= {\rm Hom} (\pi(N), \ZZ),$
we have that
$$
\sum_{i=1}^n \mu(v_i) x_i\partial_i \Phi_c (x_1, \ldots, x_n)$$
$$
=\rho(c)
\sum_{j=1}^m \mu(w_j) \big( \sum_{i \in I_j} \rho (v_i) x_i \big) \partial_j 
\Psi_{\pi(c)}
(\sum_{i \in I_1} \rho(v_i) x_i, \ldots,
\sum_{i \in I_m} \rho(v_i) x_i)
$$
$$
=\mu(\beta-c) \Phi_c (x_1, \ldots, x_n). 
$$
\end{proof}

Let $G$ denote the torsion part of $N.$ Since $G$ is finite abelian
group, we have that ${\rm Hom} (G, \CC^{\times}) \simeq G.$ Assume that $G$ has
order $k,$ and let
$(\rho_g)_{g \in G}$ be the corresponding set of independent
characters
in ${\rm Hom} (G, \CC^{\times}) \simeq G.$ When there is no torsion,
we set $|G|=1$ and $\rho_1=1.$
The characters $\rho_g$
can be easily 
extended to become multiplicative characters of $N$
by imposing that they take the value $1$ on the
free part of $N,$ after a choice of splitting. 
Under this convention, we will
view the characters $\rho_g$ as elements in 
${\rm Hom} (N, \CC^{\times}).$ As in formula $(\ref{def:mapchar}),$
for each $g \in G,$
we define the linear surjective
maps $p_g : \CC^n \to \CC^m$ by
$$
p_g (x_1, \ldots, x_n):=
(\sum_{i \in I_1} \rho_g (v_i) x_i, \ldots,
\sum_{i \in I_m} \rho_g (v_i) x_i).
$$

Let $U \subset \CC^m$ be a nonempty open set in $\CC^m$ with
the property that there exists an open set $V$ in $\RR^m$
such that
\begin{equation}\label{setarg}
\begin{split}
U=\{ &(z_1, \ldots, z_m) \, : \, (\log |z_1|, \ldots, \log |z_m|) \in V, \\
&(\arg z_1, \ldots, \arg z_m) \in (-\pi, \pi) \times \ldots \times (-\pi, \pi) \} 
\end{split}
\end{equation}
for a choice of the argument functions $(\arg z_1, \ldots, \arg z_m) \in \RR^m.$
For such a set $U \subset \CC^m,$ we have that:
\begin{lemma}\label{lemma:int}
$
\cap_{g \in G} \
p_g^{-1} (U) \not= \emptyset. 
$
\end{lemma}
\begin{proof} For each set of indices $I_j,$ choose exactly one
$i_j \in I_j$ and a complex number $x_{i_j}$ 
such that 
$$(\log |x_{i_1}|, \ldots, \log |x_{i_m}| )
\in V \setminus \{ 0 \}$$ 
and, for all $j, 1 \leq j \leq m,$  
$$\arg x_{i_j} + \pi < 2 \pi/|G|.$$
From (\ref{setarg}), we see that
$\arg (\rho_g (v_{i_j}) x_{i_j}) \in (-\pi, \pi),$
for any $g \in G$ and $1 \leq j \leq m,$ hence
$$
(\rho_g (v_{i_1}) x_{i_1}, \ldots, \rho_g (v_{i_m}) x_{i_m} ) \in U,
$$
for any $g \in G.$ By continuity, 
it is now possible to choose all the other complex numbers $x_i,
1 \leq i \leq n,$ for those indices $i$ different from any of the $i_j$'s, in 
a small enough neighborhood of the origin in the complex plane such that
$p_g (x_1, \ldots, x_n) \in U,$ for any $g \in G.$
The lemma follows.
\end{proof}

\begin{theorem}\label{thm:tor} 
Let $\Psi^{\lambda}_{\pi(K)}, \lambda \in \Lambda,$
be a set of linearly independent analytic solutions
on an open set $U \subset \CC^m$ satisfying condition (\ref{setarg}) 
to the better behaved GKZ associated to $\beta \in N \otimes \CC$
defined by $(w_1, \ldots, w_m)$ in $\pi(N).$
The associated set of $|\Lambda| \cdot |G|$ functions 
$\Phi^{\lambda, g}_K,$ $\lambda \in \Lambda, g \in G,$ is
a set of linearly independent analytic solutions
on the non-empty open set 
$\cap_{g \in G} \ p_g^{-1} (U)
\subset \CC^n$
to the better behaved GKZ associated to $\beta \in N \otimes \CC$
defined by
$(v_1, \ldots, v_n)$ in $N.$
\end{theorem}
\begin{proof} Assume that there exists constants $\alpha_{\lambda,g}$ such that
$$
\sum_{\lambda \in \Lambda, g \in G} \alpha_{\lambda,g} \Phi^{\lambda, g}_c (x) =0,$$
for any 
$c \in K, x \in \cap_{g \in G} \ p_g^{-1} (U).$ It follows that
$$
\sum_{\lambda \in \Lambda, g \in G} \alpha_{\lambda,g} \rho_g (c+c_h)
\Psi^{\lambda}_{\pi(c)} (p_g(x))=0,$$
for any 
$c \in K, x \in \cap_{g \in G} \ p_g^{-1} (U),$
and $c_h \in K$ such that $\pi(c_h)= 0.$

For each fixed $c \in K,$ we have $|G|$ linear relations indexed by $h \in G.$
The orthogonality relations for the
characters of the representations of the finite group $G$ imply that
$$
\sum_{\lambda \in \Lambda} \alpha_{\lambda,g}
\Psi^{\lambda}_{\pi(c)} (p_g(x))=0,$$
for any $g \in G, 
c \in K$ and $x \in \cap_{g \in G} \ p_g^{-1} (U).$
The linear independence of the analytic 
functions $\Psi^{\lambda}_{\pi(K)}, \lambda \in \Lambda,$
in $U$ shows that $\alpha_{\lambda,g}$ are all zero.  
\end{proof}


\section{The better behaved GKZ system with $\beta=0$ and the relation to the 
mixed Hodge structures}\label{added}

In this section we study the \emph{restriction map} 
from the space of solutions of $\GKZ(\Aa, K; \beta)$
to that of $\GKZ(\Aa,K^\circ;\beta)$ given by forgetting $\Phi_c$ for $c\in \partial K$. Our focus is primarily on the case $\beta=0$ which differs significantly from the general case. Our results are closely related to (and can be in principle derived from)  the work of Batyrev in \cite{Batyrev} on mixed Hodge structures of hypersurfaces in 
algebraic tori but our treatment is much more direct and algebraic.

\

We start with elementary examples which show that dimension of the image of the restriction 
map depends on $\beta$.
\begin{example} \label{5.1}
Let $K=\ZZ_{\geq 0}$ and $v_1=1$. Then the solution space to 
$\GKZ(\Aa, K; \beta)$ is one-dimensional, generated by $(\Phi_k),$ 
with $\Phi_k=
(\frac{d}{dx_1})^k x_1^{\beta}$.
In particular,
when $\beta=0$, the restriction to the solutions of $\GKZ(\Aa,K^\circ;\beta)$ is zero.
\end{example}

\begin{example}
Consider $K\subseteq \ZZ^3$ generated by 
$$v_1=(0,0,1),v_2=(0,1,1), v_3=(1,1,1),
v_4=(1,0,1).$$
The solutions to 
$\GKZ(\Aa, K; \beta)$ are determined uniquely by $\Phi_{(0,0,0)}$.
When $\beta=0$, they are given by 
$\Phi_{(0,0,0)}(x_1,\ldots,x_4) = c_1 + c_2 \ln(\frac {x_1x_3}{x_2x_4})$
and thus restrict to zero solutions of $\GKZ(\Aa,K^\circ;\beta)$.
\end{example}

Note that in the above examples the maps $R(f,K^\circ)\to R(f,K)$ are zero.
The main result of this section is that the dimension of the image of the restriction map
coincides with the dimension of the image of $R(f,K^\circ)\to R(f,K)$ in 
the $\beta=0$ 
case. Our approach uses the dual description of the solution space which we now
introduce.

The semigroup ring $\CC[K]$ is endowed with the structure of the module over 
${\rm Sym}^*(M_\CC)$ induced by the multiplication by the logarithmic derivatives 
of $f$. Specifically, the multiplication by the linear function $\mu:N\to\CC$ is given by
\begin{equation}\label{usual}
\mu[n] = \sum_i x_i\mu(v_i)[n+v_i].
\end{equation}
Consider a different module structure on $\CC[K]$, where the multiplication 
is modified to be 
\begin{equation}\label{tilde}
\mu\widehat{[n]} = \sum_i x_i\mu(v_i)\widehat{[n+v_i]} + \mu(n-\beta)\widehat{[n]}
\end{equation}
where we use the $\widehat{~}$ notation to avoid confusion with \eqref{usual}.
It is straightforward to check that $\mu_1\mu_2\widehat{[n]}=\mu_2\mu_1\widehat{[n]}$, so indeed 
$\CC[K]$ is endowed with another module structure over the ring ${\rm Sym}^*(M_\CC)$.
This module structure is no longer compatible with the multiplication in $\CC[K]$, rather the 
action is given by differential operators, see \cite[Def. 7.2]{Batyrev}.
\begin{definition}
We will denote the space $\CC[K]$ with the module structure given by \eqref{tilde}
by $\widehat{\CC[K]}$, where $\beta$ should be clear from the context. The notation
$\CC[K]$ will imply the module structure of \eqref{usual}. The same convention will be
used for other subsets of $N$, for example $\widehat{\CC[K^\circ]}$ and
$\CC[K^\circ]$. We will also use $\widehat{[n]}$ to denote the basis elements of $~\widehat{~}~$ 
spaces.
\end{definition}

\begin{remark}
The $\widehat{\CC[K]}$ module is \emph{not} graded. However, it admits a filtration
$(\widehat{\CC[K]}_{\leq k})$ by 
the degree so that the associated graded module is naturally isomorphic to $\CC[K]$.
\end{remark}

Recall that $\CC[K]$ is a free graded module over ${\rm Sym}^*(M_\CC)$ for nondegenerate $f$,
see \cite[Theorem 4.8]{Batyrev}. The following proposition shows that 
 $\widehat{\CC[K]}$ is also free.
\begin{proposition}\label{basis}
The module $\widehat{\CC[K]}$ is free for any $\beta$. Moreover, suppose we 
have a homogeneous basis $(u_l=\sum_{n}r_{n,l} [n])$  of $\CC[K]$ as a free module 
over ${\rm Sym}^*(M_\CC)$ and have chosen $w_l\in \widehat{\CC[K]} $ made from monomials of 
degree lower than the degree
 of $u_l$. Then the elements $(\sum_{n}r_{n,l} \widehat{[n]}+w_l)$ freely generate   $\widehat{\CC[K]}$.
\end{proposition}

\begin{proof}
If $(\sum_{n}r_{n,l} \widehat{[n]}+w_l)$ are linearly dependent over ${\rm Sym}^*(M_\CC)$,
consider the top degrees of the corresponding linear combination. Then we get a contradiction 
with the linear independence of $(\sum_{n}r_{n,l} [n])$. To show that  $(\sum_{n}r_{n,l} \widehat{[n]}+w_l)$
generate $\widehat{\CC[K]}$ over ${\rm Sym}^*(M_\CC)$, proceed by induction on degree.
\end{proof}

\begin{remark}\label{rem:basis}
It is easy to see that under the assumptions of Proposition \ref{basis}, we get an isomorphism
between $\CC[K]$ and $\widehat{\CC[K]}$, obtained by identifying the corresponding 
basis elements, which preserves the filtrations
$\widehat{\CC[K]}_{\leq \bullet}$ and $\CC[K]_{\leq \bullet}$. In fact, it is
an identity on the associated graded spaces (i.e. $\widehat{[n]}\mapsto {[n]}\mod \CC[K]_{\deg<\deg(n)}$)
and any isomorphism with this property 
is described by the above proposition.
\end{remark}

We will denote by $I$ the irrelevant ideal of the graded ring ${\rm Sym}^*{M_\CC}$.
Recall that, according to Definition \ref{def},  $\CC[K]/I\CC[K]$ is the
logarithmic Jacobian ring $R(f,K)$ associated to $f$ and $K.$
\begin{corollary}
The image of $\widehat{ \CC[K]}_{\leq k}$ in $\widehat{\CC[K]}/I\widehat{\CC[K]}$  is
naturally isomorphic to $\widehat{\CC[K]}_{\leq k}/I\widehat{ \CC[K]}_{\leq k-1}$.
The filtration on $\widehat{\CC[K]}/I\widehat{\CC[K]}$  induced from the filtration on
$\widehat{\CC[K]}$ has associated graded spaces naturally isomorphic to the graded 
components of  $\CC[K]/I\CC[K]$.
\end{corollary}

\begin{proof} This follows immediately from Proposition \ref{basis}.
\end{proof}

\begin{remark}
Analogous statements regarding $\widehat{\CC[K^\circ]}$ follow along the same lines.
\end{remark}

The following result provides the link between the algebraic
structures introduced in this section and
the better behaved GKZ system.
\begin{proposition}\label{dualdesc}
Let $f$ be nondegenerate. The spaces of analytic solutions
to $\GKZ(\Aa, K; \beta)$ and $\GKZ(\Aa, K^\circ; \beta)$ in a small neighborhood of $x$ 
are naturally isomorphic to the duals of the spaces
$\widehat{\CC[K]}/I\widehat{\CC[K]}$ and $\widehat{\CC[K^\circ]}/I\widehat{\CC[K^\circ]}$
respectively. 
\end{proposition}

\begin{proof}
The statement follows immediately from Theorem \ref{key} and Remark \ref{same}.
\end{proof}

We will now focus our attention on the case $\beta=0$. 
What's remarkable about this case is that the 
$\widehat{~}$ structure induces similar structures
on $\CC[\theta],$ for all the faces $\theta$ of $K$. Specifically, we 
have the following.
\begin{proposition}
Let $\beta=0$.
For a face $\theta$ of the cone 
$K,$ consider the ${\rm Sym}^*(M_\CC)$ module  $\widehat {\CC[\theta]}$ 
given by \eqref{tilde} for $v_i\in \theta$. This module structure descends to the module 
structure for ${\rm Sym}^*((\CC\theta)^\dual)$. 
This structure is also compatible with the natural surjection 
$\widehat {\CC[K]}\to \widehat {\CC[\theta]}$.
\end{proposition}

\begin{proof}
The statements follow  from the definitions. Note that
$\beta=0$ is essential. Otherwise, the module structure does not descend 
to ${\rm Sym}^*((\CC\theta)^\dual)$ because $\mu\in{\rm Ann}(\theta)$ no longer
act trivially on $\widehat {\CC[\theta]}$.
\end{proof}

The key result of this section is that among the various 
isomorphisms of Proposition \ref{basis} there exists an isomorphism of
${\rm Sym}^*(M_\CC)$ modules 
$\widehat{\CC[K]}\simeq \CC[K]$ which is compatible with the restrictions
to the faces.

\begin{theorem}\label{iso}
Let $\beta=0$. Then there exists (a non-canonical) collection of isomorphisms between
${\rm Sym}^*((\CC\theta)^\dual)$ modules 
$\widehat{\CC[\theta]}\simeq \CC[\theta]$ which commute with the surjections
$\widehat{\CC[\theta_1]}\to \widehat{\CC[\theta_2]}$ and 
$\CC[\theta_1]\to \CC[\theta_2]$ for all face inclusions $\theta_2\subseteq \theta_1$.
Moreover, these isomorphisms may be chosen to preserve the filtration and 
to be identity on the associated graded spaces.
\end{theorem}

\begin{proof}
Clearly, the statement of the theorem amounts to finding an isomorphism
$\widehat{\CC[K]}\simeq \CC[K]$  of Proposition \ref{basis} and Remark \ref{rem:basis}
with the property that 
$\widehat{[n]}$ is mapped into $\sum_{n_1}\alpha_{n_1,n}[n]$ where $\alpha_{n_1,n}$ 
are only nonzero for $n_1$ in the interior of the smallest face of $K$ that contains $n$.
The isomorphism will be constructed by induction on $\theta\subseteq K$,
inspired by the arguments of \cite{BresslerLunts}.

The base of induction $\theta=\{0\}$ is trivial. Suppose that
the desired isomorphisms have been constructed for all $\sigma\subsetneq\theta$ for
some face $\theta\subseteq K$, and that they are compatible with the face inclusions.

Consider the modules $\widehat{\CC[\partial\theta]}$ and ${\CC[\partial\theta]}$ over
${\rm Sym}^*((\CC\theta)^\dual)$
which are the cokernels of the inclusion maps $\widehat{\CC[\theta^\circ]}\to
\widehat{\CC[\theta]}$ and ${\CC[\theta^\circ]}\to
{\CC[\theta]}$ respectively. These modules admit ``cellular" resolutions
(cf. \cite[Def. 3.3]{BresslerLunts}) 
$$
\begin{array}{cccccccc}
0&\to&\widehat{\CC[\partial\theta]}&\to&
\displaystyle\bigoplus_{\stackrel{\sigma\subset\theta}{{\rm codim}(\sigma,\theta)=1}} 
\widehat{\CC[\sigma]}
&\to&
\displaystyle\bigoplus_{\stackrel{\sigma\subset\theta}{{\rm codim}(\sigma,\theta)=2}} 
\widehat{\CC[\sigma]}&\to\ldots
\\
&&&&\downarrow&&\downarrow&\\
0&\to&{\CC[\partial\theta]}&\to&
\displaystyle\bigoplus_{\stackrel{\sigma\subset\theta}{{\rm codim}(\sigma,\theta)=1}} 
{\CC[\sigma]}
&\to&
\displaystyle\bigoplus_{\stackrel{\sigma\subset\theta}{{\rm codim}(\sigma,\theta)=2}} 
{\CC[\sigma]}&\to\ldots
\end{array}
$$
with the vertical arrows being the isomorphisms constructed by the induction assumption.
Thus there is a unique isomorphism $\widehat{\CC[\partial\theta]}\to{\CC[\partial\theta]}$
that makes the diagram commute. Moreover, this isomorphism will be compatible 
with the filtration and will be an identity on the associated graded objects, since the same
is true for the vertical maps of the above diagram.

We will now show how to lift the isomorphism $\widehat{\CC[\partial\theta]}\to{\CC[\partial\theta]}$
to $\widehat{\CC[\theta]}\to{\CC[\theta]}$.
Denote  by $I$ the  irrelevant ideal of the polynomial ring ${\rm Sym}^*((\CC\theta)^\dual)$.
We have a surjection of graded vector spaces
$$
\CC[\theta]/I\CC[\theta] \to {\CC[\partial\theta]}/I{\CC[\partial\theta]}\to 0.
$$
Lift a graded basis $S$ of ${\CC[\partial\theta]}/I{\CC[\partial\theta]}$ to 
a set $S_1$ of graded elements
of $\CC[\theta]/I\CC[\theta]$. Then complete $S_1$ to a basis of  $\CC[\theta]/I\CC[\theta]$
by adding elements from a
set $S_2$ which map to zero in ${\CC[\partial\theta]}/I{\CC[\partial\theta]}$. We can
lift the sets $S_1$ and $S_2$ to subsets of homogeneous elements $S_3$ and $S_4$ 
in $\CC[\theta],$ respectively. We can do it in such a way as to ensure that $S_4$ maps to $0$ 
in $\CC[\partial\theta]$.
Indeed, for any lift, the image is in $I\CC[\partial\theta]$, which is in the image of $I\CC[\theta]$.
Thus the lifts can be adjusted to go to $0$ by subtracting elements supported on the boundary
of $\theta$. 
We now define $F_3$ and $F_4$ as graded submodules of $\CC[\theta]$ generated by $S_3$
and $S_4$ respectively. We have that $\CC[\theta]=F_3\oplus F_4$ by 
Nakayama's lemma where 
$F_4$ maps to zero in $\CC[\partial\theta]$ and $F_3$ maps surjectively
onto $\CC[\partial\theta].$

Let $S_5$ be the image of the set $S_3$ in $\CC [\partial \theta.]$ The set
$S_5$ can  be identified with a subset
$\widehat{S_5}$ in
$\widehat{\CC[\partial\theta]}$ by the isomorphism we have already established.
We then lift the set $S_3$
to a subset $\widehat{S_3}$ of $\widehat{\CC[\theta]}$ so that the elements
of $\widehat{S_3}$ have the same
leading terms as those of $S_3.$
Similarly, we consider $\widehat{S_4}$ in $\widehat{\CC[\theta]}$ which
have the same leading terms as $S_4$ and restrict to zero on the boundary.
Consider the morphism between $\CC[\theta]$ and $\widehat{\CC[\theta]}$ obtained by 
sending $S_3$ and $S_4$ to $\widehat{S_3}$ and $\widehat{S_4}$ respectively. It is an isomorphism
by Proposition \ref{basis}. It is also compatible with the isomorphism  $\CC[\partial\theta]\simeq\widehat{\CC[\partial\theta]}$ by construction. This finishes the induction step.
\end{proof}

For any $\beta$ there is a natural map from the solutions of $\GKZ(K,\beta)$ to 
the solutions of $\GKZ(K^\circ,\beta)$ obtained by looking at $\Phi_c,c\in K^\circ$.
Theorem \ref{iso} implies the following statement about the dimension of the image.
\begin{corollary}\label{r1}
The space of solutions of $\GKZ(K^\circ,0)$ that can be extended to
$\GKZ(K,0)$ has the same dimension as the image $R_1(f,K)$ of the map
$$
\CC[K^\circ]/I\CC[K^\circ] \to \CC[K]/I\CC[K]
$$
considered in \cite{BM}. In fact, for a given point $(x_i)$ this space admits a natural filtration by the order of vanishing at this point, and the associated graded space is naturally isomorphic to the dual of $R_1(f,K)$.
\end{corollary}

\begin{proof}
By Proposition \ref{dualdesc}, the dual of the restriction map in question is 
the map  
$$
\widehat{\CC[K^\circ]}/I\widehat{\CC[K^\circ]}\to 
\widehat{\CC[K]}/I\widehat{\CC[K]}$$
induced by the inclusion $\CC[K^\circ]\subset\CC[K]$.

By the argument of Theorem \ref{iso}, the isomorphism constructed therein induces 
a commutative diagram
$$
\begin{array}{ccccc}
\widehat{\CC[K]}&\to&\widehat{\CC[\partial K]}&\to& 0\\
\downarrow&&\downarrow&&\\
{\CC[K]}&\to&{\CC[\partial K]}&\to& 0\\
\end{array}
$$
where vertical maps preserve the filtration 
and reduce to the identity on the associated graded spaces.
By considering the kernels of the horizontal maps we get a commutative diagram of
${\rm Sym}^*(M_\CC)$ modules
\begin{equation}\label{compatible}
\begin{array}{ccccc}
0&\to&\widehat{\CC[K^\circ]}&\to&\widehat{\CC[ K]}\\
&&\downarrow&&\downarrow\\
0&\to&{\CC[K^\circ]}&\to&{\CC[K]}\\
\end{array}
\end{equation}
with the same property. 
This induces the commutative diagram 
$$
\begin{array}{ccc}
\widehat{\CC[K^\circ]}/I\widehat{\CC[K^\circ]}&\to&\widehat{\CC[ K]}/I\widehat{\CC[K]}\\
\downarrow&&\downarrow\\
{\CC[K^\circ]}/I{\CC[K^\circ]}&\to&{\CC[K]}/I{\CC[K]}
\end{array}
$$
which preserves the filtrations which are images of the corresponding filtrations before
taking the quotient by $I$. We have thus shown that the images of the corresponding maps
are isomorphic which implies the statement of the first sentence of this corollary.

We will now study the filtration and its associated graded object.
The image of the part of the filtration that comes from $\widehat{\CC[K^\circ]}_{\leq k}$ is
given by 
$$
\widehat{\CC[K^\circ]}_{\leq k}/(I\widehat{\CC[K]}_{\leq k-1} \cap \widehat{\CC[K^\circ]}_{\leq k}),
$$
in view of the same statement for the image of $\widehat{\CC[K^\circ]}_{\leq k}$. The associated
graded piece is then given by 
$$
\widehat{\CC[K^\circ]}_{\leq k}/(\widehat{\CC[K^\circ]}_{\leq k-1}+(I\widehat{\CC[K]}_{\leq k-1} \cap \widehat{\CC[K^\circ]}_{\leq k}))
$$
The isomorphism between this space and 
$$
{\CC[K^\circ]}_{\leq k}/({\CC[K^\circ]}_{\leq k-1}+(I{\CC[K]}_{\leq k-1} \cap{\CC[K^\circ]}_{\leq k}))=R_1(f,K)_k
$$
induced by Theorem \ref{iso} naturally lifts to the isomorphism between
$\widehat{\CC[K^\circ]}_{\leq k}/\widehat{\CC[K^\circ]}_{\leq k-1}$
and ${\CC[K^\circ]}_{\leq k}/{\CC[K^\circ]_{\leq k-1}}$ which is natural in the sense that it does not 
depend on the choices of Theorem \ref{iso}. This finishes the proof.
\end{proof}

\begin{remark}
It is in general impossible to construct the commutative diagram \eqref{compatible}
if $\beta\neq 0$. For instance, in the set up of the Example \ref{5.1} we must have 
$\widehat {[0]}\mapsto [0]$ for the right vertical arrow. However, this isomorphism
does not identify the ideals $\widehat{\CC[K^\circ]}$ and $\CC[K^\circ]$. 
\end{remark}

\begin{remark}
In place of $\CC[K^\circ]$, we may similarly consider submodules of $\widehat{\CC[K]}$ 
which are linear combinations  of monomials that are supported at faces of dimension 
at least as large as some given number. This gives a filtration on 
$\widehat{\CC[ K]}/I\widehat{\CC[K]}$. Together with the filtration by the degree these
are pretty much the weight and Hodge filtration on the cohomology 
of a hypersurface in algebraic torus, see \cite{Batyrev}.
\end{remark}

\begin{remark}
The result of Corollary \ref{r1} can be presumably deduced from \cite[section 8]{Batyrev},
but we were unable to find a precise reference. At any rate, this would have been a very roundabout 
derivation, as compared to our method based on Theorem \ref{iso}.
\end{remark}

\section{Further comments and open problems} 
One can consider a \emph{partial semigroup} version of the better behaved
GKZ hypergeometric system as follows. Let $\Sigma$ be a fan in the cone
$K$ spanned by $v_i$. We can replace the equations $\partial_i\Phi_c=\Phi_{c+v_i}$
in Definition \ref{trueGKZdef} by 
$$
\partial_i \Phi_c = \left\{\begin{array}{ll}
\Phi_{c+v_i},& {\rm there~exists~a~cone~of~} \Sigma
 {\rm ~that~contains} ~v_i~{\rm and}~c
\\
0,&{\rm otherwise.}
\end{array}\right.
$$
Then if $\Sigma$ admits a piecewise linear strictly convex function, then all the arguments
of this paper are applicable to this more general system.

\smallskip
The properties of the better behaved version of the GKZ system show that 
this system is better suited than the usual GKZ if any type of
functorial considerations are to be invoked. We expect that the mirror symmetric 
identification of the Fourier-Mukai transform and the analytic
continuation formulae for the solutions to the GKZ system discussed
in \cite{BH1} continues to hold in the better behaved GKZ case.  
In fact, the context of this paper is more natural since the
rank of the space of local solutions to the better behaved GKZ 
one side matches the rank of the orbifold cohomology/stacky $K$-theory 
on the other side. This was not  the case in our previous work. 
It would be an interesting problem to study an appropriately defined
category of better behaved GKZ systems and its functorial properties,
part of which would mirror the properties of category of toric
DM stacks. In the same realm, we expect that, in an appropriate sense,  
the system $\GKZ(\Aa, K; \beta)$ is dual to 
the system $\GKZ(\Aa, K^{\circ}; -\beta).$ 

More importantly, we believe that the better behaved GKZ system
lends itself to a process of categorification, which as a first step,
is expected to provide a non-commutative categorical resolution of a Gorenstein
toric singularity. Such a categorification will have to pass all
the toric mirror symmetric checks, and, as such, would have a transcendental
component which is missing in the algebraic proposals of 
non-commutative resolutions currently 
available in the literature. We hope to come back to this problem in 
a future paper.

\end{document}